\documentclass[11pt]{article}
\usepackage{amssymb,amsmath}
\usepackage{amsthm,,latexsym}
\usepackage{graphicx}
\usepackage{enumerate}
\usepackage{mathdots}

\newtheorem{theorem}{Theorem}[section]
\newtheorem{lemma}[theorem]{Lemma}
\newtheorem{corollary}[theorem]{Corollary}

\theoremstyle{definition}
\newtheorem{definition}[theorem]{Definition}
\newtheorem{example}[theorem]{Example}

\theoremstyle{remark}
\newtheorem{remark}{Remark}

\long\def\symbolfootnote[#1]#2{\begingroup%
\def\thefootnote{\fnsymbol{footnote}}\footnote[#1]{#2}\endgroup}

\newcommand{\be}{\begin{equation}}
\newcommand{\ee}{\end{equation}}
\newcommand{\ben}{\begin{enumerate}}
\newcommand{\een}{\end{enumerate}}

\begin{document}


\title{Binomial Polynomials Mimicking Riemann's Zeta Function}


\author{M. W. Coffey\textsuperscript{a} and M. C. Lettington\textsuperscript{b}
}





\maketitle


\begin{abstract}

The (generalised) Mellin transforms of Gegenbauer polynomials have polynomial factors $p^{\lambda}_n(s)$, whose zeros all lie on the `critical line' $\Re\,s=1/2$ (called critical polynomials). The transforms are identified in terms of combinatorial sums related to Gould's S:4/3, S:4/2 and S:3/1 binomial coefficient forms. Their `critical polynomial' factors are then identified in terms of $_3F_2(1)$ hypergeometric functions. Furthermore, we extend these results to a $1$-parameter family of polynomials with zeros only on the critical line that possess the functional relation $p_n(s;\beta)=\pm p_n(1-s;\beta)$.

Normalisation yields the rational function $q_n^{\lambda}(s)$ whose denominator has singularities on the negative real axis. Moreover as $s\rightarrow \infty$ along the positive real axis, $q_n^{\lambda}(s)\rightarrow 1$ from below.

For the Chebyshev polynomials we obtain the simpler S:2/1 binomial form, and with $\mathcal{C}_n$ the $n$th Catalan number, we deduce $4\mathcal{C}_{n-1}p_{2n}(s)$ and $\mathcal{C}_{n}p_{2n+1}(s)$ yield integers with only odd prime factors. The results touch on analytic number theory, special function theory, and combinatorics.

\end{abstract}



\symbolfootnote[0]{The authors would like to thank Dr J. L. Hindmarsh and Prof. M N Huxley for their helpful comments and suggestions.\newline
2010 \emph{Mathematics Subject Classification}: 11B65 (primary), 05A10, 33C20, 33C45, 42C05, 44A20, 30D05 (secondary).\newline
\emph{Key words and phrases}: critical polynomials, binomial coefficients, Gould combinatorial summations, Mellin transforms, hypergeometric functions.}

\section{Introduction}
\label{intro}

 The motivation for this present work is to further understand the triangle of connections that exist between binomial coefficients, functions which only have critical zeros (those on the line $\Re\, s=1/2$ or zeros on the real line, and henceforth referred to as \emph{critical polynomials}), and prime numbers.

As stated by K. Dilcher and K. B. Stolarsky, \cite{dilcher}

\vspace{2mm} ``Two of the most ubiquitous objects in mathematics are the
sequence of prime numbers and the binomial coefficients (and thus Pascal's triangle).
A connection between the two is given by a well-known characterisation of the prime
numbers: Consider the entries in the $k$th row of Pascal's triangle, without the initial and
final entries. They are all divisible by $k$ if and only if $k$ is a prime''.
\vspace{2mm}

By considering a modified form of Pascal's triangle, whose $k$th row consists of the integers
\be
a(k,j):=\frac{(2k-1)(2k+1)}{2j+3}\binom{k+j}{2j+1},\qquad k\in \mathbb{N},\qquad 0\leq j\leq k-1,
\label{eq:m1}
\ee
Dilcher and Stolarsky obtained an analogous characterisation of pairs of twin prime numbers $(2k-1,2k+1)$. This says that the entries in the $k$th row of the $a(k,s)$ number triangle are divisible by $2k-1$ with exactly one exception, and are divisible by $2k+1$ with exactly one exception, if and only if $(2k-1,2k+1)$ is a pair of twin prime numbers.

The corresponding sequence of polynomials $A_k(x)$ obtained from the $k$th row of the number triangle generated by the integers $a(k,j)$ is given by
\[
A_k(x)=\sum_{j=0}^{k-1}a(k,j)x^j=\sum_{j=0}^{k-1}\frac{(2k-1)(2k+1)}{2j+3}\binom{k+j}{2j+1}\,x^j.
\]
It was shown in \cite{dilcher} that this polynomial family satisfies the four-term recurrence relation
\[
A_{k+4}(x)=(2x+4)\left (A_{k+3}(x)+A_{k+1}(x)\right )-(4x^2+4x+6)A_{k+2}(x)-A_{k}(x),
\]
as opposed to a three-term recurrence relation required for orthogonality (e.g. see \cite{szego}), and so they do not constitute an orthogonal polynomial system.

However it is also shown in \cite{dilcher} that the polynomials $A_k(x)$ are closely linked to the orthogonal polynomial system of Gegenbauer polynomials $C_n^\lambda(x)$, defined for $\lambda>-1/2$, $\lambda\ne 0$ (e.g., \cite{andrews}), by the hypergeometric series representation \cite{nbs} (p. 773-802)
\be
C_n^\lambda(x)={{(2\lambda)_n} \over {n!}} ~_2F_1\left(2\lambda+n,-n;\lambda+{1 \over 2};
{{1-x} \over 2}\right),
\label{eq:geghyo}
\ee
(so that $C_n^\lambda(1)=(2\lambda)_n/n!$). The hypergeometric series for $C_n^\lambda(x)$ can be re-written in terms of binomial coefficients and powers of 2 such that
\be
C_n^\lambda(x)=\sum _{r=0}^{\lfloor n/2\rfloor}(-1)^r \binom{n-r}{r}\binom{n-r-1+\lambda}{n-r}(2x)^{n-2r},
\label{eq:m5}
\ee
and taking $\lambda = 2$, one finds that the relation to the $A_k(x)$ polynomial system is given by
\[
A_k(x)=C_{k-1}^2\left (\frac{x+2}{2}\right)+(x+6)C_{k-2}^2\left (\frac{x+2}{2}\right)+C_{k-3}^2\left (\frac{x+2}{2}\right).
\]
The Legendre Polynomials $P_n(x)$ are the case $\lambda = 1/2$ of the Gegenbauer polynomials $C_n^{1/2}(x)$, and a close connection between these polynomials, the prime numbers and the absolute value of the Riemann zeta function, $\zeta(s)=\sum_{n=1}^\infty \frac{1}{n^s}$, defined for $\Re(s)>1$, was established in \cite{choi1}, where $|\zeta(s)|$ is expressed as an infinite sum over products of Legendre polynomials and functions derived from prime numbers.

The location of the zeros of the Riemann zeta function is famously known as the Riemann Hypothesis (1859), which states that all of the non-trivial zeros of $\zeta(s)$ (the trivial zeros lie at the negative even integers) lie on the critical line $\Re\,s=1/2$. In 1901 von Koch reinforced the connection between $\zeta(s)$ and the prime numbers, demonstrating that the Riemann Hypothesis is equivalent to the statement that the error term for $\pi(x)$, the number of primes up to $x$, is of order of magnitude $O\left (\sqrt{x}\log(x)\right )$ \cite{koch}. Riemann had originally shown that
\[
\pi(x)\sim {\rm Li}(x)+\sum_{n=2}^\infty \frac{\mu(n)}{n}{\rm Li}(x^{1/n}),\qquad\text{where}\qquad
{\rm Li}=\int_2^x\frac{{\rm d}u}{\log{u}},
\]
and where $\mu(n)$ is the M\"obius function, which returns $0$ if $n$ is divisible by a prime squared and $(-1)^k$ if $n$ is the product of $k$ distinct primes.

The B\'{a}ez-Duarte equivalence to the Riemann Hypothesis \cite{baez,maslanka} links the the Riemann Hypothesis (and so the prime numbers) to binomial coefficients, via the infinite sequence of real numbers $c_t$, defined such that
\[
c_t:=\sum_{s=0}^t(-1)^s\binom{t}{s}\frac{1}{\zeta(2s+2)},
\]
with the assertion that the Riemann hypothesis is true if and only if $c_t = O(t^{-3/4 + \epsilon})$, for integers $t\geq 0$, and for all $\epsilon >0$.

In relation to understanding the triangle of connections that exist between the three objects consisting of the prime numbers, the binomial coefficients, and functions which only have critical zeros, it is those between the binomial coefficients and the `critical polynomials' that appears to be the least studied, thus motivating the results contained in this paper.

Before elaborating further, we mention some standard notation in which $_2F_1$ denotes the Gauss
hypergeometric function, $_pF_q$ the generalized hypergeometric function,
\[
(a)_n=\Gamma(a+n)/\Gamma(a)=(-1)^n {{\Gamma(1-a)} \over {\Gamma(1-a-n)}}
\]
the Pochhammer symbol,
with $\Gamma$ the Gamma function \cite{andrews,bailey,grad}. We also define $\varepsilon=0$ for $n$ even and $\varepsilon=1$ for $n$ odd.

Our starting point is a special case of the more general class of integrals
\[
\int_0^{a}x^{\alpha-1}(a^2-x^2)^{\beta-1}C_{2n+\epsilon}^{\lambda}(cx)\,{\rm d}x
=\frac{(-1)^n(\lambda)_{n+\epsilon}c^\epsilon a^{\alpha+2\beta+\epsilon-2}}{2^{1-\epsilon}n!}B\left (\frac{\alpha+\epsilon}{2},\beta\right )
\]
\be
\times _3F_2\left (-n,n+\lambda+\epsilon,\frac{\alpha+\epsilon}{2};\epsilon +\frac{1}{2},\frac{\alpha+\epsilon+2\beta}{2};a^2 c^2\right ),
\label{gr1}
\ee
where $\epsilon = (1-(-1)^n)/2$; $a,\Re\,\beta>0;\Re\,\alpha >-\epsilon$, and $B(x,y)=\frac{\Gamma(x)\Gamma(y)}{\Gamma{(x+y)}}$, is the beta function (see entry 2.21.2(1) on p. 517 of \cite{grad}), and which applies to the following definition and Theorem \ref{thm:gr6}.

\begin{definition}
For $\lambda>-1/2$, we define the generalised Mellin transform $M_n^\lambda(s)$, such that
\be
M_n^\lambda(s)=\int_0^1 {{C_n^\lambda(x)x^{s-1}} \over {(1-x^2)^{3/4-\lambda/2}}}dx
=\int_0^{\pi/2} \cos^{s-1} \theta ~C_n^\lambda(\cos \theta) \sin^{\lambda-1/2} \theta ~d\theta,
\label{eq:m9}
\ee
wherein $x=\cos{\theta}$, and we assume that $\Re\,s>0$ for $n$ even and $\Re\,s>-1$ for $n$ odd, denoting by $p^\lambda_n(s)$ the polynomial factor of $M^\lambda_n(s)$.
Then for $\lambda=1$ we have the Mellin transform of the Chebyshev functions \cite{mason,rivlin} of the second kind given by
\be
M_n(s) \equiv \int_0^1 x^{s-1} U_n(x) {{dx} \over {(1-x^2)^{1/4}}}.
\label{eq:m10}
\ee
\end{definition}


In \cite{bumpng,coffeymellin,coffeyxi}, Mellin transforms were used on $[0,\infty)$.
Here we consider Mellin transformations for functions supported on $[0,1]$,
$$({\cal M}_0f)(s)=\int_0^1 f(x)x^s {{dx} \over x}.$$
For properties of the Mellin transform, we mention \cite{butzer}.

Our main results show that the polynomial factors $p_n^\lambda(s)$ of the Mellin transforms in (\ref{eq:m9}) of the Gegenbauer (and so Chebyshev) functions $C_n^\lambda(x)$, yield families of `critical polynomials' $p_n^\lambda(s)$, $n=0,1,2,\ldots$, of degree $\lfloor n/2\rfloor $, satisfying the functional equation $p_n^\lambda(s)=(-1)^{\lfloor n/2\rfloor} p_n^\lambda(1-s)$. Additionally we find that (up to multiplication by a constant) these polynomials can be written explicitly as variants of Gould S:4/1 and S:3/2 binomial sums (see \cite{gould}), the latter form being
\be
p_{2n}^\lambda(s)=\textstyle{n! (2 n)! \binom{n+\lambda-1}{n}\binom{n+\frac{1}{2}(s+\lambda)-\frac{3}{4}}{n}}
\sum _{r=0}^n \frac{ (-1)^{n-r} 2^{2 r-1}\binom{n+r+\lambda-1}{r}\binom{n+r}{2r}\binom{\frac{1}{2}(s-2)+r}{r}}{\binom{n+r}{r}\binom{\frac{1}{2}(s+\lambda)-\frac{3}{4}+r}{r}},
\label{eq:crite}
\ee
\be
p_{2n+1}^\lambda(s)=\textstyle{n! (2 n+1)! \binom{n+\lambda}{n+1}\binom{n+\frac{1}{2}(s+\lambda)-\frac{1}{4}}{n}}
\sum _{r=0}^n \frac{ (-1)^{n-r}
2^{2 r}\binom{n+r+\lambda}{r}\binom{n+r+1}{2r+1}\binom{\frac{1}{2}(s-1)+r}{r}}{\binom{n+r+1}{r}\binom{\frac{1}{2}(s+\lambda)-\frac{1}{4}+r}{r}}.
\label{eq:crito}
\ee
In the case of the Chebyshev polynomials ($\lambda=1)$, this simplifies to the $S:2/1$ form, due to cancellation of binomial factors, and with $\mathcal{C}_n=\frac{1}{n+1}\binom{2n}{n}$, the $n$th Catalan number, $s$ an integer, we show that polynomials $4\mathcal{C}_{n-1}p_{2n}(s)$
and $\mathcal{C}_{n}p_{2n+1}(s)$ yield integers with only odd prime factors.
\begin{example}\label{example:p}
The first few transformed polynomials $p_n^\lambda(s)$, are given by
\begin{align*}
p_0^\lambda(s)=&1/2,\\
p_1^\lambda(s)=&\lambda,\\
p_2^\lambda(s)=&\frac{1}{4} \lambda (2 \lambda+1) (2 s-1)=\frac{1}{2} \lambda (2 \lambda+1) \left(s-\frac{1}{2}\right),\\
p_3^\lambda(s)=&\frac{1}{2} \lambda (\lambda+1) (2 \lambda+1) (2 s-1)= \lambda (\lambda+1) (2 \lambda+1) \left(s-\frac{1}{2}\right),\\
p_4^\lambda(s)=&\frac{1}{8} \lambda (\lambda+1) (2 \lambda+1) \left(8 \lambda s^2-8 \lambda s+6 \lambda+12s^2-12 s+15\right),\\
=&\frac{1}{8} \lambda (\lambda+1) (2 \lambda+1)\textstyle{\left(s-\left (\frac{1}{2}-
\frac{ i\sqrt{9+9 \lambda+2 \lambda^2}}{ 3+2 \lambda}\right )\right)
 \left(s-\left (\frac{1}{2}+\frac{ i\sqrt{9+9 \lambda+2 \lambda^2}}{ 3+2 \lambda}\right )\right)},\\
p_5^\lambda(s)=&\frac{1}{4} \lambda (\lambda+1) (\lambda+2) (2 \lambda+1) \left(8 \lambda s^2-8 \lambda s+14
   \lambda+12 s^2-12 s+51\right)\\
    =&\frac{1}{4} \lambda (\lambda+1) (\lambda+2) (2 \lambda+1)\textstyle{\left(s-\left (\frac{1}{2}-
\frac{ i\sqrt{3}\sqrt{2 \lambda^2+11 \lambda+12}}{ 3+2 \lambda}\right )\right)\left(s+\left (\frac{1}{2}-
\frac{ i\sqrt{3}\sqrt{2 \lambda^2+11 \lambda+12}}{ 3+2 \lambda}\right )\right)}
\end{align*}
\end{example}
The `critical polynomials' under consideration here, in a sense motivate the Riemann hypothesis, and have many important applications to analytic number theory.  For example, using the Mellin transforms of Hermite functions, Hermite polynomials multiplied by a Gaussian factor, Bump and Ng \cite{bumpng} were able to generalise Riemann's second proof of the functional equation of the zeta function $\zeta(s)$, and to obtain a new representation for it.

The polynomial factors of the Mellin transforms of Bump and Ng are realised as certain $_2F_1(2)$ Gauss hypergeometric functions \cite{coffeymellin}.
In a different setting, the polynomials $p_n(x)= ~_2F_1(-n,-x;1;2)=(-1)^n ~_2F_1(-n,x+1;1;2)$ and $q_n(x)=i^n n! p_n(-1/2-ix/2)$ were studied \cite{kirsch}, and they directly correspond to the Bump and Ng polynomials with $s=-x$.  Kirschenhofer, Peth\"{o}, and Tichy considered combinatorial properties of $p_n$, and developed Diophantine properties of them.  Their analytic results for $p_n$ include univariate and bivariate generating functions, and that its
zeros are simple, lie on the line $x=-1/2+it$, $t \in \mathbb{R}$, and that its zeros
interlace with those of $p_{n+1}$ on this line.  We may observe that these polynomials
may as well be written as $p_n(x)={{n+x} \choose n} ~_2F_1(-n,-x;-n-x;-1)$, or
$$p_n(x)={{(-1)^n 2^n \Gamma(n-x)} \over {n!\Gamma(-x)}} ~_2F_1\left(-n,-n;x+1-n;{1 \over 2}
\right).$$
Previous results obtained by the authors related to this area of research are discussed in \cite{coffeylettunpub,coffeygen,HDF}, where in the former paper families of `critical polynomials' are obtained from generalised Mellin transforms of classical orthogonal Legendre polynomials. In \cite{coffeygen}
the first author has addressed the Mellin transform of certain generalised Hermite functions, where the resulting critical polynomials possess
a reciprocity relation. In the latter paper sequences of `critical polynomials' are considered which can also be obtained by generalised Mellin transforms of families of orthogonal polynomials whose coefficients are the weighted binomial coefficients defined by
\[
B_k(x)=\sum_{j=0}^{k}b(k,j)x^j=\sum_{j=0}^{k}\frac{2k+1}{2j+1}\binom{k+j}{2j}\,x^j.
\]
Here we find that for $\Re\,s>-1/4$, the Mellin transforms
\[
M_n^B(s)= \int_{-4}^0 {B_n(x)x^{s-3/4} \over {(4+x)^{3/4}}}dx=(-1)^{s+5/4}4^s 4^{-n-1} \Gamma(1/4) p_n(s) {{\Gamma\left(s+{1 \over 4}\right)} \over {\Gamma\left(s+{{2n+1} \over 2}\right)}},
\]
yield critical polynomial factors $p_n(s)$, which obey the \emph{perfect reflection} functional equation $p_n(s)=\pm p_n(~1~-~s~)$.

The `perfect-reflection' functional equation $p_n^\lambda(s)=(-1)^{\lfloor n/2\rfloor} p_n^\lambda(1-s)$, is similar to that for Riemann's xi function
$\xi(s)$, defined by
$\xi(s) = \tfrac{1}{2} s(s-1) \pi^{-s/2} \Gamma\left(\tfrac{1}{2} s\right) \zeta(s)$, and which satisfies $\xi(s)=\xi(1-s)$, so that for $t\in\mathbb{R}$, the zeros of $\xi(1/2+it)$ and $\zeta(1/2+it)$ are identical. Drawing upon this analogy, one interpretation is that the polynomials $p_n^\lambda(s)$ are normalised (from a functional equation perspective) polynomial forms of the rational functions $q_n^\lambda(s)$, defined for $n\in\mathbb{N}$ by
\be
q_{2n}^\lambda(s)=\frac{2\,p_{2n}^\lambda(s)}{\textstyle\lambda  (n-1)! (2 n)!\binom{2n+2\lambda-1}{2n-1} \binom{n+\frac{1}{2}(s+\lambda)-\frac{3}{4}}{n}}
=\frac{2^{2n+1} p_{2n}^\lambda(s)}{(2\lambda)_{2n}\prod_{j=1}^n (2s+2\lambda +4j-3)},
\label{eq:norme}
\ee
\be
q_{2n+1}^\lambda(s)=\frac{p_{2n+1}^\lambda(s)}{\textstyle\lambda (n)! (2 n)!\binom{2n+2\lambda}{2n} \binom{n+\frac{1}{2}(s+\lambda)-\frac{1}{4}}{n}}
=\frac{2^{2n+1} p_{2n+1}^\lambda(s)}{(2\lambda)_{2n+1}\prod_{j=1}^n (2s+2\lambda +4j-1)},
\label{eq:normo}
\ee
where both numerator and denominator polynomials of $q_n^\lambda(s)$ are of degree $\lfloor n/2\rfloor$.

For $\lambda>-1/2$, $\lambda \ne 0$, and $\Re\,s>0$, the $\lfloor n/2\rfloor$ linear factors of the denominator polynomials of $q_n^\lambda(s)$, are each non-zero, so that for these values of $\lambda$, we have $q_n^\lambda(s)$ has no singularities with $\Re\,s>0$. Hence
the rational function $q_n^\lambda(s)$ has the same `critical zeros' as the polynomial $p_n^\lambda(s)$, and for $t\in\mathbb{R}$, the roots of $p_n^\lambda(1/2+it)$ and $q_n^\lambda(1/2+it)$ are identical.

The $[n/2]$ poles of $q_{n}^\lambda(s)$ (so zeros of the denominator polynomial of $q_{n}^\lambda(s)$) occur on the negative real axis when $2s=3-\lambda -4j$ or $2s=1-2\lambda -4j$ (depending on the parity of $n$) and in Theorem \ref{theorem5}, for $s\in(1,\infty)$, we find that $q_{n}^\lambda(s)$ takes values on $(0,1)$, with
$\mathop{\lim}_{s\rightarrow \infty}q_n^\lambda(s)=~1$ (from below), and obeys the functional equation
\be
q_{n}^\lambda(s)=(-1)^{\lfloor n/2\rfloor}\binom{\lfloor n/2\rfloor +\frac{1-s+\lambda+\epsilon}{2}-\frac{3}{4}}{\lfloor n/2\rfloor}
\binom{\lfloor n/2\rfloor+\frac{s+\lambda+\epsilon}{2}-\frac{3}{4}}{\lfloor n/2\rfloor}^{-1}q_{n}^\lambda(1-s).
\label{eq:fq}
\ee
It follows that on $\mathbb{R}_{> 1}$, the behaviour of $q_{n}^\lambda(s)$ has similarities to that of $1/\zeta(s)$, albeit with a rate of convergence to the limit point 1, considerably slower than for $1/\zeta(s)$.

To give an overview, the present work is split into four sections, with the main results concerning the critical polynomials arising from Mellin transforms of Gegenbauer polynomials appearing after this introduction in the second section. In the third section we prove these results, utilising continuous Hahn polynomials to locate the `critical zeros'. The fourth and concluding section then considers further possible extensions to these results.

\section{Critical Polynomial Results}

\begin{theorem}\label{thm:gr6}
The Mellin transforms (\ref{eq:m9}) may be written as $_3F_2(1)$ hypergeometric functions, such that
\be
M_{2n}^\lambda(s)
=(-1)^n M_0^\lambda(s)\binom{\lambda+n-1}{n} \, _3F_2\left(-n,\lambda+n,\frac{s}{2};\frac{1}{2},\frac{\lambda}{2}+\frac{s}{2}+\frac{1}{4};1\right),
\label{gr2}
\ee
and
\be
M_{2n+1}^\lambda(s)=(-1)^n 2 M_0^\lambda(s+1) (n+1) \binom{\lambda+n}{n+1} \,
_3F_2\left(-n,\lambda+n+1,\frac{s}{2}+\frac{1}{2};\frac{3}{2},\frac{\lambda+s}{2}+\frac{3}{4};1\right),
\label{gr3}
\ee
where
 \be
 M_0^\lambda(s)={{\Gamma\left({\lambda \over 2}+{1 \over 4}\right)\Gamma\left({s \over 2}\right)}
\over {2\Gamma \left({{s+\lambda} \over 2}+{1 \over 4}\right)}},\qquad
\text{\rm and}\qquad M_1^\lambda(s)=2\lambda M_0^\lambda(s+1),
\label{gr5}
\ee
and the latter identities for odd and even cases give the equivalent hypergeometric form, which can be written for general $n$ as
\be
M_n^\lambda(s)={{\Gamma\left({\lambda \over 2}+{1 \over 4}\right)\Gamma\left({{s+n} \over 2}\right)}\over {2\Gamma \left({{s+n+\lambda} \over 2}+{1 \over 4}\right)}}\binom{2\lambda+n-1}{n}
~_3F_2\left({\lambda \over 2}+{1 \over 4},{{1-n} \over 2},-{n \over 2};{1 \over 2}+\lambda,
1-{{(n+s)} \over 2};1\right).
\label{gr4}
\ee
\end{theorem}

 \begin{theorem}
\label{theorem2}
Let $M_n^\lambda(s)$ be defined as in $(\ref{eq:m9})$. Then we have

\noindent
(a) the mixed recurrence relation
$$n M_n^\lambda(s)=2(\lambda+n-1)M_{n-1}^\lambda(s+1)-(2\lambda+n-2)M_{n-2}^\lambda(s),$$

\noindent
(b) the generating function
$$G^\lambda(s,t) = \sum_{k=0}^\infty M_k^\lambda(s)t^k =\int_0^1 {1 \over {(1-x^2)^{3/4-\lambda/2}}}{x^{s-1} \over {(1-2tx+t^2)^\lambda}}dx$$
$$={1 \over {(1+t^2)^\lambda}}{{\Gamma\left({1 \over 4}+{\lambda \over 2}\right)} \over 2}\left[{{\Gamma(\lambda)\Gamma\left({s \over 2}\right)} \over {\Gamma\left({{s+\lambda} \over 2}+{1 \over 4}\right)}} ~_3F_2\left({{\lambda+1} \over 2},{\lambda \over 2},{s \over 2};{1 \over 2},{{s+\lambda} \over 2}+{1 \over 4};{{4t^2} \over {(1+t^2)^2}}\right) \right.$$
$$\left. + {{2t\Gamma(\lambda+1)} \over {(1+t^2)}}{{\Gamma\left({{s+1} \over 2}\right)} \over {\Gamma\left({{s+\lambda} \over 2}+{3 \over 4}\right)}} ~_3F_2\left({{\lambda+1} \over 2},
1+{\lambda \over 2},{{s+1} \over 2};{3 \over 2},{{s+\lambda} \over 2}+{3 \over 4}; {{4t^2} \over {(1+t^2)^2}}\right) \right], $$

\noindent (c) the polynomial factors satisfy the functional equation $p_n^\lambda(s)=(-1)^{\lfloor n/2 \rfloor} p_n^\lambda(1-s)$,

\noindent
(d) the recurrence relation in $s$
$$[6-4(\lambda+2\lambda n+n^2)-16s+8s(s+1)]M_n^\lambda(s)$$
$$+[-9+4(n+\lambda)^2+16(s+2)-4(s+2)(s+3)]M_n^\lambda(s+2)$$
$$-4(s-1)(s-2)M_n^\lambda(s-2)=0,$$
with
$$M_0^\lambda(s)={{\Gamma\left({\lambda \over 2}+{1 \over 4}\right)\Gamma\left({s \over 2}\right)}
\over {2\Gamma \left({{s+\lambda} \over 2}+{1 \over 4}\right)}},\qquad
\text{\rm and}\qquad M_1^\lambda(s)=2\lambda M_0^\lambda(s+1),$$
and
\noindent
(e) (location of zeros) the polynomial factors  $p_n^\lambda(s)$ of $M_n^\lambda(s)$, have zeros only on the critical line.
\end{theorem}

\begin{lemma}
\label{lemma7}
Let $\lambda=1$, so that $M_n(s)$ is the Mellin transform of the Chebyshev function given in (\ref{eq:m10}). Then we have that

(a)
$$M_n(s)=(n+1){{\Gamma\left({3 \over 4}\right)\Gamma\left({{n+s} \over 2}\right)} \over {2\,\Gamma\left({{n+s} \over 2}+{3 \over 4}\right)}} ~_3F_2\left({3 \over 4},{{1-n} \over 2},-{n \over 2};{3 \over 2},1-{{n+s} \over 2};1\right),$$
and (b)
$$M_n(s)=2^{n-1}{{\Gamma\left({3 \over 4}\right)\Gamma\left({{n+s} \over 2}\right)} \over {\Gamma\left({{n+s} \over 2}+{3 \over 4}\right)}} ~_3F_2\left({{1-n} \over 2},-{n \over 2},{1 \over 4}-{{(n+s)} \over 2};-n,1-{{n+s} \over 2};1\right).$$
\end{lemma}

\begin{proof}
The interchange of finite summation and the integration of (\ref{eq:m10}) is used in
both instances.
\end{proof}


\begin{theorem}
\label{theorem1} (a) When $\lambda=1$, the polynomials $p_n$, corresponding to the Chebyshev polynomials of the second kind, satisfy the simplified recursion relation, with~$p_0=\Gamma(3/4)/2$~and~$p_1=\Gamma(3/4)$.
for $n$ even,
\be
p_n(s)=sp_{n-1}(s+1)-{1 \over 2}\left(s+n-{1 \over 2}\right)p_{n-2}(s), \label{eq:1.2a}
\ee
and for $n$ odd,
\be
p_n(s)=2p_{n-1}(s+1)-{1 \over 2}\left(s+n-{1 \over 2}\right)p_{n-2}(s). \label{eq:1.2b}
\ee
(b) The polynomials $p_n(s)$, of degree $\lfloor n/2 \rfloor$, satisfy the functional equation
$$p_n(s)=(-1)^{\lfloor n/2 \rfloor} p_n(1-s).$$

\noindent
(c) These polynomials have zeros only on the
critical line.  Further, all zeros $\neq 1/2$ occur in complex conjugate pairs.
\end{theorem}

\begin{theorem}
\label{theorem3}
The Mellin transforms (\ref{eq:m9}) may be written as a constant multiplied by a variant on Gould's combinatorial S:4/2 and S:3/1 functions, such that
\[
M_{2n}^\lambda(s)=M_0^\lambda(s)\sum _{r=0}^n {\frac{(-1)^{n-r} 2^{2 r}\binom{n+r+\lambda -1}{n+r} \binom{n+r}{2 r}
\binom{\frac{s-2}{2}+r}{r}\binom{n+\frac{s+\lambda }{2}-\frac{3}{4}}{n-r} }{\binom{n}{r}\binom{n+\frac{s+\lambda }{2}-\frac{3}{4}}{n}}}
\]
\[
=M_0^\lambda(s)\sum _{r=0}^n \frac{(-1)^{n-r}  2^{2 r}
\binom{n+r+\lambda-1}{n+r}\binom{n+r}{2r}\binom{\frac{s-2}{2}+r}{r}}{\binom{\frac{s+\lambda}{2}-\frac{3}{4}+r}{r}},
\]
and
\[
M_{2n+1}^\lambda(s)=M_0^\lambda(s+1)\sum_{r=0}^n \frac{(-1)^{n-r} 2^{2 r+1}\binom{n+r+\lambda}{n+r+1} \binom{n+r+1}{2r+1}\binom{\frac{s-1}{2}+r}{r}\binom{n+\frac{s+\lambda}{2}-\frac{1}{4}}{n-r}}{\binom{n}{r}\binom{n+\frac{s+\lambda}{2}-\frac{1}{4}}{n}}
\]
\[
=M_0^\lambda(s+1)\sum _{r=0}^n \frac{(-1)^{n-r}2^{2 r+1}\binom{n+r+\lambda}{n+r+1} \binom{n+r+1}{2 r+1}\binom{\frac{s-1}{2}+r}{r}}{\binom{\frac{s+\lambda}{2}-\frac{1}{4}+r}{r}}.
\]
\end{theorem}

\begin{corollary}
\label{corollary1}
When $s=2t$ is an even positive integer then we have
\[
M_0^\lambda(s)=\frac{1}{4t}\binom{\frac{\lambda}{2}+\frac{1}{4}+t-1}{t}^{-1},
\]
and when $s=2t+1$ is an odd positive integer
\[
M_1^\lambda(s)=2\lambda M_0^\lambda(s+1)=2\lambda M_0^\lambda(2t+2)=\frac{\lambda}{2t+2}\binom{\frac{\lambda}{2}+\frac{1}{4}+t}{t+1}^{-1}.
\]
Hence, when $s$ is a non-negative integer then we have the Gould's S:4/3 summation variants
\[
M_{2n}^\lambda(2t)=\frac{1}{2t}\sum _{r=0}^n \frac{(-1)^{n-r} 2^{2 r}\binom{n+r+\lambda -1}{n+r}\binom{n+r}{2 r}\binom{t-1+r}{r}
\binom{n+t+\frac{\lambda }{2}-\frac{3}{4}}{n-r} }{\binom{n}{n-r}\binom{t+\frac{\lambda}{2}-\frac{3}{4}}{t} \binom{n+t+\frac{\lambda }{2}-\frac{3}{4}}{n}},
\]
and
\[
M_{2n+1}^\lambda(2t+1)=\frac{\lambda}{t+1}\sum _{r=0}^n \frac{(-1)^{n-r} 2^{2 r+1}\binom{n+r+\lambda }{n+r+1}\binom{n+r+1}{2 r+1} \binom{t+r}{r}
\binom{n+t+\frac{\lambda }{2}+\frac{1}{4}}{n-r}}{\binom{n}{r}\binom{t+\frac{\lambda}{2}+\frac{1}{4}}{t+1} \binom{n+t+\frac{\lambda }{2}+\frac{1}{4}}{n}}.
\]
\end{corollary}

\begin{theorem}
\label{theorem4}
Let $\varepsilon=(1-(-1)^n)/2$. Then the Mellin transforms are of the form
$$M_n^\lambda(s)={{\Gamma\left({\lambda \over 2}+{1 \over 4}\right)\Gamma\left({{s+\epsilon}
\over 2}\right)} \over {2(n!)\Gamma\left({{s+n+\lambda} \over 2}+{1 \over 4}\right)}}\,
p_n^\lambda(s),$$
 and the polynomial factors $p_n^\lambda(s)$ satisfy the difference equations in $s$
\[
(s+2)  (2 \lambda+4 n-2 s-5) p_{2 n}^{\lambda }(s+4)- \left(8 \lambda n+2 \lambda+8 n^2-4 s^2-12 s-11\right) p_{2 n}^{\lambda }(s+2)
\]
\[
-(s+1) (2 \lambda+4 n+2s+1) p_{2 n}^{\lambda }(s)=0,
\]
and
\[
(s+3) (2 \lambda+4 n-2   s-3) p_{2 n+1}^{\lambda }(s+4)-\left(8 \lambda n+6 \lambda+8 n^2+8 n-4 s^2-12 s-9\right)
p_{2 n+1}^{\lambda}(s+2)
\]
\[
-s (2 \lambda+4 n+2 s+3) p_{2 n+1}^{\lambda }(s)=0.
\]

\end{theorem}
.

\begin{theorem}[Binomial `critical polynomial' theorem]
\label{theorem5}
(a) The polynomials $p_n^\lambda(s)$, can be written (up to multiplication by a constant) in terms of binomial coefficients and powers of 2, as a variant of a Gould S:4/1 combinatorial function, such that
\[
p_{2n}^\lambda(s)=n! (2 n)! \sum _{r=0}^n \frac{ (-1)^{n-r} 2^{2 r-1}\binom{n+r+\lambda-1}{n+r} \binom{n+r}{2r}
\binom{\frac{1}{2}(s-2)+r}{r}\binom{n+\frac{1}{2}(s+\lambda)-\frac{3}{4}}{n-r}}{\binom{n}{r}},
\]
\[
p_{2n+1}^\lambda(s)=n! (2 n+1)! \sum _{r=0}^n \frac{(-1)^{n-r} 2^{2 r}\binom{n+r+\lambda}{n+r+1}
 \binom{n+r+1}{2 r+1}\binom{\frac{1}{2}(s-1)+r}{r}\binom{n+\frac{1}{2}(s+1+\lambda)-\frac{3}{4}}{n-r}}{\binom{n}{r}},
\]
or as a Gould S:3/2 combinatorial function variant, such that
\[
p_{2n}^\lambda(s)=n! (2 n)! \binom{n+\lambda-1}{n}\binom{n+\frac{1}{2}(s+\lambda)-\frac{3}{4}}{n}
\sum _{r=0}^n \frac{ (-1)^{n-r} 2^{2 r-1}\binom{n+r+\lambda-1}{r}\binom{n+r}{2r}\binom{\frac{1}{2}(s-2)+r}{r}}{\binom{n+r}{r}\binom{\frac{1}{2}(s+\lambda)-\frac{3}{4}+r}{r}},
\]
\[
p_{2n+1}^\lambda(s)=n! (2 n+1)! \binom{n+\lambda}{n+1}\binom{n+\frac{1}{2}(s+\lambda)-\frac{1}{4}}{n}
\sum _{r=0}^n \frac{ (-1)^{n-r}
2^{2 r}\binom{n+r+\lambda}{r}\binom{n+r+1}{2r+1}\binom{\frac{1}{2}(s-1)+r}{r}}{\binom{n+r+1}{r}\binom{\frac{1}{2}(s+\lambda)-\frac{1}{4}+r}{r}},
\]
thus establishing the binomial `critical polynomial' relationships (\ref{eq:crite}) and (\ref{eq:crito}).


\noindent
(b) Let $q_{n}^\lambda(s)$ denote the rational function in $s$ derived from the S:3/2 form of the `critical polynomial' $p_n^\lambda(s)$ such that
\[
q_{2n}^\lambda(s)=\frac{2n\binom{n+\lambda-1}{n}}{\lambda\binom{2n+2\lambda-1}{2n-1}}\sum _{r=0}^n \frac{ (-1)^{n-r} 2^{2r-1}\binom{n+r+\lambda-1}{r}\binom{n+r}{2r}\binom{\frac{1}{2}(s-2)+r}{r}}{\binom{n+r}{r}\binom{\frac{1}{2}(s+\lambda)-\frac{3}{4}+r}{r}},
\]
\[
q_{2n+1}^\lambda(s)=\frac{(2n+1)\binom{n+\lambda}{n+1}}{\lambda\binom{2n+2\lambda}{2n}}\sum _{r=0}^n \frac{ (-1)^{n-r}
2^{2 r}\binom{n+r+\lambda}{r}\binom{n+r+1}{2r+1}\binom{\frac{1}{2}(s-1)+r}{r}}{\binom{n+r+1}{r}\binom{\frac{1}{2}(s+\lambda)-\frac{1}{4}+r}{r}}.
\]
Then we have
\[
q_{2n}^\lambda(s)=\frac{2\,p_{2n}^\lambda(s)}{\textstyle\lambda  (n-1)! (2 n)!\binom{2n+2\lambda-1}{2n-1} \binom{n+\frac{1}{2}(s+\lambda)-\frac{3}{4}}{n}}
=\frac{2^{2n+1} p_{2n}^\lambda(s)}{(2\lambda)_{2n}\prod_{j=1}^n (2s+2\lambda +4j-3)},
\]
\[
q_{2n+1}^\lambda(s)=\frac{p_{2n+1}^\lambda(s)}{\textstyle\lambda (n)! (2 n)!\binom{2n+2\lambda}{2n} \binom{n+\frac{1}{2}(s+\lambda)-\frac{1}{4}}{n}}
=\frac{2^{2n+1} p_{2n+1}^\lambda(s)}{(2\lambda)_{2n+1}\prod_{j=1}^n (2s+2\lambda +4j-1)},
\]
where both numerator and denominator polynomials of $q_n^\lambda(s)$ are of degree $\lfloor n/2\rfloor$.

\noindent
(c) For $\lambda >-1/2$, $\lambda \ne 0$, and $\Re\,s>0$, the rational function $q_n^\lambda(s)$ has no singularities, and has the same `critical zeros' as the polynomial $p_n^\lambda(s)$, so that for $t\in\mathbb{R}$, the roots of $p_n^\lambda(1/2+it)$ and $q_n^\lambda(1/2+it)$ are identical.

When $s\in\mathbb{R}_{>1}$, $q_{n}^\lambda(s)$ takes values on $(0,1)$, with
$\mathop{\lim}_{s\rightarrow \infty}q_n^\lambda(s)=~1$ (from below), and with $\varepsilon=0$ for $n$ even and $\varepsilon=1$ for $n$ odd, obeys the functional equation
\be
q_{n}^\lambda(s)=(-1)^{\lfloor n/2\rfloor}\binom{\lfloor n/2\rfloor +\frac{1-s+\lambda+\epsilon}{2}-\frac{3}{4}}{\lfloor n/2\rfloor}
\binom{\lfloor n/2\rfloor+\frac{s+\lambda+\epsilon}{2}-\frac{3}{4}}{\lfloor n/2\rfloor}^{-1}q_{n}^\lambda(1-s).
\label{eq:fq1}
\ee
As $s\rightarrow \infty$ along the positive real axis it follows that $q_{n}^\lambda(s)\rightarrow 1$, from below, as does $1/\zeta(s)$.
\end{theorem}
\begin{corollary}
\label{corollary40}
The polynomial factors arising from the Mellin transform of the Chebyshev polynomials have the simpler form
as a variant of a Gould S:2/1 combinatorial function, such that
\[
p_{2n}(s)=n! (2 n)! \binom{n+\frac{s}{2}-\frac{1}{4}}{n}
\sum _{r=0}^n \frac{ (-1)^{n-r} 2^{2 r-1}\binom{n+r}{2r}\binom{\frac{1}{2}(s-2)+r}{r}}{\binom{\frac{s}{2}-\frac{1}{4}+r}{r}},
\]
\[
p_{2n+1}(s)=n! (2 n+1)! \binom{n+\frac{s}{2}+\frac{1}{4}}{n}
\sum _{r=0}^n \frac{ (-1)^{n-r}
2^{2 r}\binom{n+r+1}{2r+1}\binom{\frac{1}{2}(s-1)+r}{r}}{\binom{\frac{s}{2}+\frac{1}{4}+r}{r}}.
\]
For $s =t\in \mathbb{Z}$ an integer, and with $\mathcal{C}_n=\frac{1}{n+1}\binom{2n}{n}$, the $n$th Catalan number, the polynomials $4\mathcal{C}_{n-1}p_{2n}(t)$
and $\mathcal{C}_{n}p_{2n+1}(t)$ yield integers with only odd prime factors. Moreover the polynomials
\be
\label{intexam}
\frac{2^{2n+1}}{(2n)!}p_{2n}(t),\quad\text{and}\quad \frac{2^{2n+1}\mathcal{T}_{n+1}}{(2n+2)!}p_{2n+1}(t),
\ee
with $\mathcal{T}_{n+1}$ the largest odd factor of $n+1$, yield odd integers with fewer prime factors.
\end{corollary}

\begin{example} As an example we give respectively tables of the expressions in $(\ref{intexam})$, for $0\leq n\leq5$, with $t$ an integer in the range $-4\leq t\leq4$. The functional relation can be clearly observed in both cases.
\[
\begin{array}{c|ccccccccc}
{\bf n \backslash \,t} & {\bf-4} & {\bf 3} &{\bf -2} & {\bf-1} &{\bf 0} &{\bf 1} &{\bf 2} &{\bf 3} &{\bf 4}\\ \hline
{\bf 0}& 1 & 1 & 1 & 1 & 1 & 1 & 1 & 1 & 1  \\
{\bf 1} &-27 & -21 & -15 & -9 & -3 & 3 & 9 & 15 & 21  \\
{\bf 2} & 421 & 261 & 141 & 61 & 21 & 21 & 61 & 141 & 261  \\
{\bf 3} & -7119 & -3969 & -1995 & -861 & -231 & 231 & 861 & 1995 & 3969  \\
{\bf 4} & 154665 & 80361 & 36729 & 13401 & 3465 & 3465 & 13401 & 36729 & 80361  \\
{\bf 5} & -4029795 & -1946637 & -828135 & -293073 & -65835 & 65835 & 293073 & 828135 & 1946637  \\
\end{array}
\]
\[
\begin{array}{c|ccccccccc}
{\bf n \backslash \,t} & {\bf-4} & {\bf 3} &{\bf -2} & {\bf-1} &{\bf 0} &{\bf 1} &{\bf 2} &{\bf 3} &{\bf 4} \\ \hline
{\bf 0} & 1 & 1 & 1 & 1 & 1 & 1 & 1 & 1 & 1  \\
{\bf 1} & -9 & -7 & -5 & -3 & -1 & 1 & 3 & 5 & 7  \\
{\bf 2} & 279 & 183 & 111 & 63 & 39 & 39 & 63 & 111 & 183  \\
{\bf 3} & -1341 & -819 & -465 & -231 & -69 & 69 & 231 & 465 & 819  \\
{\bf 4} & 128637 & 72765 & 37581 & 17325 & 8157 & 8157 & 17325 & 37581 & 72765  \\
{\bf 5} & -1809459 & -959805 & -465975 & -197505 & -52731 & 52731 & 197505 & 465975 & 959805  \\
\end{array}
\]
\end{example}

\begin{theorem}[Perfect reflection property theorem]
\label{theorem6}
We say that $f(s)$ has the `perfect reflection property' to mean $f(\overline{s})=\overline{f(s)}$, $f(s)=\chi f(1-s)$, with $\chi=\pm 1$,
$f(s)=0$, only when $\Re\, s=1/2$.

Then the polynomials
$$p_n(s;\beta)={{\Gamma\left({{n+s} \over 2}\right)} \over {\Gamma\left({{s+\varepsilon} \over 2}\right)}}
~_3F_2\left(1-\beta,{{1-n} \over 2},-{n \over 2};2(1-\beta),1-{{(n+s)} \over 2};1\right),$$
have the \emph{perfect reflection property} with $\chi(n)=(-1)^{\lfloor n/2\rfloor}$,
wherein $\varepsilon=0$ for $n$ even and $=1$ for $n$ odd, $\beta<1$,
of degree $\lfloor n/2 \rfloor$, satisfy the functional equation
$p_n(s;\beta)=(-1)^{\lfloor n/2 \rfloor} p_n(1-s;\beta)$.  These polynomials have zeros only
on the critical line, and all zeros $\neq 1/2$ occur in complex conjugate pairs.
\end{theorem}

\begin{corollary}
\label{corollary2}
(a) The properties of Theorem \ref{theorem6} are satisfied by the polynomials
$$p_n(s;0)={{2(n+s)} \over {(n+1)(n+2)}}{{\Gamma\left({{n+s} \over 2}\right)} \over {\Gamma\left({{s+\varepsilon} \over 2}\right)}} \left[1- ~_2F_1\left(-{{(n+1)} \over 2},-{n \over
2}-1;-{{(n+s)} \over 2};1\right)\right]$$
$$={{2(n+s)} \over {(n+1)(n+2)}}{{\Gamma\left({{n+s} \over 2}\right)} \over {\Gamma\left({{s+\varepsilon} \over 2}\right)}} \left[1- {{\Gamma\left(-{{(n+s)} \over 2}\right)} \over {\Gamma\left({{1-s} \over 2}\right)}} {{\Gamma\left({{n+3-s} \over 2}\right)} \over {\Gamma\left(1-{s \over 2}\right)}}\right].$$
(b) More generally, for $\beta$ a negative integer, the properties of Theorem \ref{theorem6} are satisfied
by the polynomials $p_n(s;-m)$, and these polynomials may be written in terms of elementary
factors and the Gamma function.
\end{corollary}

\section{Proof of the Main Results}

\begin{proof}[Proof of Theorem \ref{thm:gr6}]
Setting $a=1$, $\alpha=s$, $\beta = \frac{\lambda}{2}+\frac{1}{4}$, and $c=1$ in (\ref{gr1}), and taking $\epsilon=0$ and then $\epsilon=1$ respectively gives us (\ref{gr2}) and (\ref{gr3}). The displays in (\ref{gr5}) then follow by setting $n=0$ in (\ref{gr2}) and (\ref{gr3}).

To see the hypergeometric form (\ref{gr4}) of $M_n^\lambda(s)$, we use the series representation (\cite{rainville}, p. 278 (6))
$$C_n^\lambda(x)=\sum_{k=0}^{[n/2]} {{(2\lambda)_n x^{n-2k} (x^2-1)^k} \over {4^k k! (\lambda+
1/2)_k(n-2k)!}},$$
and recall a form of the Beta integral,
$$\int_0^1 x^{a-1}(1-x^2)^{b-1}dx={1 \over 2}B\left({a \over 2},b\right), ~~~~\mbox{$\Re\,a>0$},
~~\mbox{$\Re\,b>0$}.$$
Then
$$M_n^\lambda(s)=\sum_{k=0}^{[n/2]} {{(2\lambda)_n (-1)^k} \over {4^k k! (\lambda+
1/2)_k(n-2k)!}}\int_0^1x^{s+n-2k-1}(1-x^2)^{k+\lambda/2-3/4}dx$$
$$=\sum_{k=0}^{[n/2]} {{(2\lambda)_n (-1)^k} \over {4^k k! (\lambda+
1/2)_k(n-2k)!}}{1 \over 2}B\left({{s+n} \over 2}-k,k+{\lambda \over 2}+{1 \over 4}\right)$$
$$={{(2\lambda)_n} \over {2\Gamma\left({{s+n+\lambda} \over 2}+{1 \over 4}\right)}}
\sum_{k=0}^{[n/2]} {{(-1)^k} \over {4^k k!}}{{\Gamma\left({{s+n} \over 2}-k\right)
\Gamma\left(k+{\lambda \over 2}+{1 \over 4}\right)} \over {(\lambda+1/2)_k(n-2k)!}}$$
$$={{\Gamma\left({\lambda \over 2}+{1 \over 4}\right)\Gamma\left({{s+n} \over 2}\right)}\over {2\Gamma \left({{s+n+\lambda} \over 2}+{1 \over 4}\right)}}
~\binom{2\lambda+n-1}{n}~_3F_2\left({\lambda \over 2}+{1 \over 4},{{1-n} \over 2},-{n \over 2};{1 \over 2}+\lambda,
1-{{(n+s)} \over 2};1\right),$$
as required, where the above has used the duplication formula for the Gamma function so that
$${1 \over {(n-2k)!}}={1 \over {\Gamma(n-2k+1)}}={4^k \over {n!}}\left(-{n \over 2}\right)_k
\left({{1-n} \over 2}\right)_k.$$

 \end{proof}

\begin{proof}[Proof of Theorem \ref{theorem2}]  (a) follows from (\cite{andrews}, p. 303 or \cite{grad}, p. 1030)
$$(n+2)C_{n+2}^\lambda(x)=2(\lambda+n+1)xC_{n+1}^\lambda(x)-(2\lambda+n)C_n^\lambda(x),$$
$C_0^\lambda(x)=1$, and $C_1^\lambda(x)=2\lambda x$.
(b) follows from (\cite{andrews}, p. 302 or \cite{grad}, p. 1029)
$$(1-2xt+t^2)^{-\lambda}=\sum_{n=0}^\infty C_n^\lambda(x)t^n.$$
(c) The Mellin transforms may be computed explicitly by several means.  In particular, from
$$C_n^\lambda(x)={{(\lambda)_n} \over {n!}} (2x)^n ~_2F_1\left(-{n \over 2},{{1-n} \over 2};
1-n-\lambda;{1 \over x^2}\right),$$
transformation of the $_2F_1$ function to argument $x^2$ and use of \cite{grad} (p. 850), we find that
$$M_n^\lambda(s)={{(\lambda)_n} \over {n!}} 2^{n-1}\sqrt{\pi}(-i)^n \Gamma\left({\lambda \over
2}+{1 \over 4}\right)\Gamma(1-\lambda-n)$$
$$\times\left\{{{-2i\Gamma\left({{s+1} \over 2}\right)} \over {\Gamma\left({1 \over 2}-
\lambda-{n \over 2}\right)}} {{_3F_2\left({{1-n} \over 2},\lambda+{{n+1} \over 2},{{s+1}
\over 2};{3 \over 2},{3 \over 4}+{{\lambda+s} \over 2};1\right)} \over
{\Gamma\left(-{n \over 2}\right)\Gamma\left({{\lambda+s} \over 2}+{3 \over 4}\right)}}\right.$$
$$\left. + {{\Gamma\left({s \over 2}\right)} \over {\Gamma\left({{1-n} \over 2}\right)}} {{_3F_2\left(-{n \over 2},\lambda+{n \over 2},{s \over 2};{1 \over 2},{1 \over 4}+{{\lambda+s} \over 2};1\right)} \over
{\Gamma\left(1-\lambda-{n \over 2}\right)\Gamma\left({{\lambda+s} \over 2}+{1 \over 4}\right)}}\right \}.$$
The second line of the right member provides the transform for $n$ odd and the third
line for $n$ even.  Then transformation of the $_3F_2$ functions and arguments similar to the
proof of Theorem~\ref{theorem1} (b) may be used to show that the degree of the polynomial factors of the transforms is again $\lfloor n/2 \rfloor$, and that they satisfy the functional equation $p_n^\lambda(s)=(-1)^{\lfloor n/2 \rfloor} p_n^\lambda(1-s)$.

An alternative proof of the functional equation can be deduced directly from the first remark, concerning the fact that $B_m=0$ in the recurrence relation, following on from the proof of part (e) of this theorem.

(d) To obtain the difference equation for $M_n^\lambda(s)$, we apply the ordinary differential equation satisfied by Gegenbauer polynomials (e.g.,
\cite{grad}, p. 1031)
$$(x^2-1)y''(x)+(2\lambda+1)xy'(x)-n(2\lambda+n)y(x)=0.$$
If $f(x)\equiv C_n^\lambda(x)/(1-x^2)^{3/4-\lambda/2}$, we then substitute
$C_n^\lambda(x)=(1-x^2)^{3/4-\lambda/2}f(x)$ into this differential equation.
We then find that
$${1 \over 4}(1-x^2)^{-1/4-\lambda/2}\left[(6-4(\lambda+2\lambda n+n^2)+(-9+4(\lambda+n)^2)x^2
)f(x) \right.$$
$$\left. +4(x^2-1)(-4xf'(x)+(1-x^2)f''(x))\right]=0.$$
It follows that the quantity in square brackets is zero.  We multiply it by $x^{s-1}$
and integrate from $x=0$ to $1$, integrating the $f'$ term once by parts, and the $f''$
term twice by parts.  We determine that the Mellin transforms satisfy the following
difference equation:
$$[6-4(\lambda+2\lambda n+n^2)-16s+8s(s+1)]M_n^\lambda(s)$$
$$+[-9+4(n+\lambda)^2+16(s+2)-4(s+2)(s+3)]M_n^\lambda(s+2)$$
$$-4(s-1)(s-2)M_n^\lambda(s-2)=0,$$
and hence the result.

(e) To show that the resulting zeros of $p_n^\lambda(s)$ occur only on $\Re\,s=1/2$ we apply a connection with continuous Hahn polynomials, which are given by (e.g., \cite{andrews} p. 331, \cite{askey})
\be
\label{hahnhyp}
h_m(x;a,b,c,d)=i^m {{(a+c)_m(a+d)_m} \over {m!}} ~_3F_2\left(-m,m+a+b+c+d-1,a+ix;a+c,a+d;1
\right).
\ee
By using the second transformation of a terminating $_3F_2(1)$ series given in the Appendix,
we have
$$h_m(x;a,b,c,d)={i^m \over {m!}}(a+b+c+d+m-1)_m(1-b-m-ix)_m $$
$$ \times ~_3F_2(1-b-c-m,1-b-d-m,-m;2-a-b-c-d-2m,1-b-m-ix;1).$$
Then comparing with the $_3F_2(1)$ function of Theorem \ref{thm:gr6}
\[
M_n^\lambda(s)={{\Gamma\left({\lambda \over 2}+{1 \over 4}\right)\Gamma\left({{s+n} \over 2}\right)}\over {2\Gamma \left({{s+n+\lambda} \over 2}+{1 \over 4}\right)}}\binom{2\lambda+n-1}{n}
~_3F_2\left({\lambda \over 2}+{1 \over 4},{{1-n} \over 2},-{n \over 2};{1 \over 2}+\lambda,
1-{{(n+s)} \over 2};1\right),
\]
we see when $n=2m$ is even, that setting
\[
x=\frac{i}{2}\left (-s+\frac{1}{2}\right ),\quad a=c=\frac{1}{2}-\frac{\lambda}{2}-m,\quad b=d=\frac{1}{4},
\]
our polynomial
factors $p_{n}^\lambda(s)$ are proportional to continuous Hahn polynomials such  that
\be p_{n}^\lambda(s)= (m!)^2 2^{2 m-1} (-i)^m \binom{m+\lambda-1}{m} h_m\left(\frac{-i}{2}\left (s-\frac{1}{2}\right );\frac{1}{2}-\frac{\lambda}{2}-m,{1 \over 4},\frac{1}{2}-\frac{\lambda}{2}-m, {1 \over 4}\right)\label{hahn1}.
\ee
Similarly when $n=2m+1$ is odd, setting
\[
x=\frac{i}{2}\left (-s+\frac{1}{2}\right ),\quad a=c=-\frac{\lambda}{2}-m,\quad b=d=\frac{3}{4},
\]
our polynomial factors $p_{n}^\lambda(s)$ are proportional to continuous Hahn polynomials such that
\be p_{n}^\lambda(s)= (m!)^2 2^{2 m} (-i)^m \lambda\binom{m+\lambda}{m} h_m\left(\frac{-i}{2}\left (s-\frac{1}{2}\right );-{\lambda \over 2}-m,\frac{3}{4},-{\lambda \over 2}-m, \frac{3}{4}\right)\label{hahn2}.\ee
For fixed values of $a,b,c,d$, the continuous Hahn polynomials are an orthogonal system of polynomials which satisfy the recurrence relation (e.g see 6.10.11 of \cite{andrews})
\[
A_{m}\hat{h}_{m+1}(x)=((a+ix)+A_{m}+C_{m})\hat{h}_{m}(x)-C_{m}\hat{h}_{m-1}(x),
\]
where
$$
\hat{h}_{m}(x)\ :=\hat{h}_{m}(x;a,\ b,\ c,\ d)=D_m h_{m}(x)=\ \frac{m!}{i^{m}(a+c)_{m}(a+d)_{m}}h_{m}(x;a,\ b,\ c,\ d)
$$
so that $D_m=m!/(i^{m}(a+c)_{m}(a+d)_{m})$, $\hat{h}_{m}(x)$ is the $_3F_2$ hypergeometric function given in (\ref{hahnhyp}), and
\[
A_{m}=\textstyle{-\frac{(m+a+b+c+d-1)(m+a+c)(m+a+d)}{(2m+a+b+c+d-1)(2m+a+b+c+d)}},\quad
C_{m}= \textstyle{\frac{m(m+b+c-1)(m+b+d-1)}{(2m+a+b+c+d-2)(2m+a+b+c+d-1)}}.
\]
In the case that $a=\bar{c}$ and $b=\bar{d}$, the recurrence relation simplifies to
\[
A_{m}\hat{h}_{m+1}(x)=ix\hat{h}_{m}(x)-C_{m}\hat{h}_{m-1}(x),
\]
and substituting for $\hat{h}_m$ and rearranging we have
\be
\label{eq:fundrec}
h_{m+1}(x)=\frac{i x D_m}{D_{m+1}A_m}h_m(x)-\frac{C_m D_{m-1}}{D_{m+1}A_m}h_{m-1}(x)
=G_m\, x\,h_m(x)-H_m\, h_{m-1}(x),
\ee
say, where
\[
G_m=\frac{ i D_m}{D_{m+1}A_m},\quad\text{and}\quad H_m=\frac{C_m D_{m-1}}{D_{m+1}A_m}.
\]
As the $a,b,c,d$ values must be constant for the conditions of the orthogonality theorems to be met, we set $m=u$ a constant in the variable $a$, so that in (\ref{eq:fundrec}) we take $a=c=\frac{1}{2}-\frac{\lambda}{2}-u,\quad b=d=\frac{1}{4}$ to obtain $h_{m+1}(x)=G_m\, x\,h_m(x)-H_m\, h_{m-1}(x)$, with
\be
\label{eq:gh1}
G_m=\frac{ (-2 \lambda +4 m-4 u+1) (-2 \lambda +4 m-4 u+3)}{2 (m+1) (-2 \lambda +2 m-4 u+1)},
\ee
and
\be
\label{gh2}
H_m=\frac{(2 m-1) (2 \lambda -4 m+4 u+1) (-2 \lambda +4 m-4 u+3) (-\lambda +m-2 u)}{16 (m+1) (-2 \lambda +2 m-4 u+1)}.
\ee
Similarly setting $a=c=-\frac{\lambda}{2}-u,\quad b=d=\frac{3}{4}$ in (\ref{eq:fundrec}) gives us
$h_{m+1}(x)=G_m\, x\,h_m(x)-H_m\, h_{m-1}(x)$, with $G_m$ the same as in (\ref{eq:gh1}), but where $H_m$ is now given by
\be
\label{gh3}
H_m=\frac{(2 m+1) (2 \lambda -4 m+4 u+1) (-2 \lambda +4 m-4 u+3) (-\lambda +m-2 u-1)}{16 (m+1) (-2 \lambda +2 m-4 u+1)}.
\ee
For the above coefficient $G_m$, and the two choices for the coefficient $H_m$, both of the resulting recurrence relations $h_{m+1}(x)=G_m\, x\,h_m(x)-H_m\, h_{m-1}(x)$ are of the form
\[
P_{m+1}(x)=(A_m x+B_m)P_m(x)-C_m P_{m-1}(x),
\]
with $B_m=0$, which is the form of the recurrence relation satisfied by a system of orthogonal polynomials that are not monic (e.g. see (4.2) p19 of \cite{chihara}).

More rigorously, if $\tilde{P}_m(x)$ is the monic (scaled) polynomial corresponding to $P_m(x)$, so that $P_m(x)=k_m\tilde{P_m}$, where $k_{-1}=1$, then in generality one finds that (e.g. see (4.3) p19 of \cite{chihara})
\[
A_m=\frac{k_{m+1}}{k_m},\quad,B_m=-c_{m+1} \frac{k_{m+1}}{k_m},\quad C_m=\mu_n\frac{k_{m+1}}{k_{m-1}}=\mu_n A_mA_{m-1},
\]
so that $c_m=0$, $\mu_m=C_m/(A_mA_{m-1})$, and where the corresponding monic recurrence relation can be written as
\be
\label{eq:fund}
\tilde{P}_m(x)=(x-c_m)\tilde{P}_{m-1}(x)-\mu_m\tilde{P}_{m-2}(x),\quad \tilde{P}_{-1}(x)=0,\quad \tilde{P}_0(x)=1,\quad m=1,2,3,\ldots
\ee
For our two recurrence relations we find that $\mu_m=H_m/(G_mG_{m-1})$, so that in the case $n=2m$ is even we have
\be
\mu_m=\frac{m (2 m-1) (2 \lambda -2 m+4 u+1) (-\lambda +m-2 u)}{4 (-2 \lambda +4 m-4 u-3) (-2 \lambda +4 m-4 u+1)},
\label{mu1}
\ee
and when $m=2m+1$ is odd
\be
\mu_m=\frac{m (2 m+1) (2 \lambda -2 m+4 u+1) (-\lambda +m-2 u-1)}{4 (-2 \lambda +4 m-4 u-3) (-2 \lambda +4 m-4 u+1)}.
\label{mu2}
\ee
For fixed values of $u$ and $\lambda$, and $m=0,1,2,3,\ldots$, with $\mu_m\neq 0$ and well defined, we obtain a pair of families of Hahn polynomials given by
\be
\label{hfam1}
h_m\left(x;\frac{1}{2}-\frac{\lambda}{2}-u,{1 \over 4},\frac{1}{2}-\frac{\lambda}{2}-u, {1 \over 4}\right),
\ee
and
\be
\label{hfam2}
h_m\left(x;-{\lambda \over 2}-u,\frac{3}{4},-{\lambda \over 2}-u, \frac{3}{4}\right).
\ee
In  both (\ref{hfam1}) and (\ref{hfam2}), we find that $h_{-1}(s)=0$ and $h_{0}(s)=1$, and applying Favard's Theorem (see for example \cite{chihara} Theorem 4.4, p21) concerning polynomial sequences satisfying the three-term recurrence relation given in (\ref{eq:fund}), for $c_m\in\mathbb{R}$ and $\mu_m\neq 0$ it follows that (\ref{hfam1}) and (\ref{hfam2}) form two families of orthogonal polynomial systems.

It is a well known fact (e.g. \cite{chihara}, Section 5, Chapter 1, or \cite{rivlin,szego}) that systems of orthogonal polynomials have only real zeros which interlace. Hence the polynomial families corresponding to (\ref{hfam1}) and (\ref{hfam2}) have only real zeros which interlace on the real line. Setting $x=\left (\frac{-i}{2}\right )(s-\frac{1}{2})$ in (\ref{hfam1}) and (\ref{hfam2}), the resulting polynomials will therefore have their zeros dilated by the factor of $\frac{1}{2}$ on the real line; rotated by $\frac{\pi}{2}$ clockwise onto the imaginary axis by the factor of $-i$, and then translated by $-\frac{1}{2}$ from the imaginary axis to the critical line $\Re s=\frac{1}{2}$. This means that setting $x=\left (\frac{-i}{2}\right )(s-\frac{1}{2})$ in (\ref{hfam1}) and (\ref{hfam2}), results in families of polynomials which have zeros only on the critical line $\Re\,s=\frac{1}{2}$.

Restricting $u\in\mathbb{N}$, to be a positive integer, analysis of (\ref{hfam1}) and (\ref{hfam2}) shows that for $\lambda\neq\frac{1}{2}$ (the Legendre polynomial case proven by the authors in \cite{coffeylettunpub}) and $\lambda\neq\frac{3}{2}$, $\mu_m$ is well defined for $u\leq m$, and when $\lambda =1$, $\mu_m$ is well defined for $u<2m$. If $\lambda$ is not an integer or half-integer then $\mu_m$ is well defined for all $u\in\mathbb{N}$.

Therefore, for each integer value $u\leq m$, and fixed $\lambda>-\frac{1}{2}$, $\lambda\neq\frac{1}{2}$, $\lambda\neq\frac{3}{2}$, we get a family of Hahn polynomials in (\ref{hfam1}), and also in (\ref{hfam2}). We can associate with each Hahn polynomial a lattice point $(m,u)$ in $\mathbb{R}^2$ in the positive quadrant of the plane, so that an orthogonal polynomial family associated with a specific value of $u$ generated by (\ref{hfam1}), corresponds to lattice points on the line $m=u$, parallel to the horizontal axis. For each particular family (so parallel line), there will exist one value of $m$, namely $m=u$, where replacing the polynomial variable with
$\left (\frac{-i}{2}\right )(s-\frac{1}{2})$, the Hahn polynomial corresponds to that given for $p_n^{\lambda}(s)$ given in (\ref{hahn1}). Hence when $n=2m$ is even, the family of polynomials $p_n^{\lambda}(s)$ can be visualised as points on a line diagonally cutting through the parallel lines corresponding to the different families of Hahn polynomials with the polynomial variable change $x=\left (\frac{-i}{2}\right )(s-\frac{1}{2})$. Conversely, for each family of Hahn polynomials determined by the integer $u$, there will exist exactly one polynomial in the family which is proportional to a $p_n^{\lambda}(s)$ polynomial. A similarly argument holds in the odd case $n=2m+1$, relating the $p_n^{\lambda}(s)$ polynomials to families of Hahn polynomials generated by (\ref{hfam2}) with the polynomial variable change $x=\left (\frac{-i}{2}\right )(s-\frac{1}{2})$.

The above argument implies that all the zeros of our polynomial families $p_n^{\lambda}(s)$ lie on the critical line $\Re s=\frac{1}{2}$, as required.

\end{proof}

\begin{remark} By Theorem 4.3 of \cite{chihara}, as $c_m=0$ in the recurrence (\ref{eq:fund}), it follows that $h_m(-x)=(-1)^m h_m(x)$, thus giving another proof of the functional relation $p_n^\lambda(s)=(-1)^{\lfloor n/2\rfloor}p_n^\lambda(1-s)$.
\end{remark}

\begin{remark} From the above geometric interpretation of the $p_n^\lambda$ polynomial line diagonally cutting through the parallel Hahn polynomial lines, it follows that scaled $p_n^\lambda(s)$ will not satisfy the three term recurrence relation obeyed by the Hahn polynomials.
\end{remark}

\begin{remark} If a family of polynomials with only critical zeros whose distribution of zeros is proportional to that of the Riemann zeta function is identified, then if possible it might be of interest to apply the above arguments and see for which values the recurrence coefficients satisfy Favard's Theorem.
\end{remark}

\begin{proof}[Proof of Theorem \ref{theorem1}] (a) According to Lemma \ref{lemma7},
the Mellin transforms are of the form 
\be
\label{eq:2.1}
M_n(s)=\Gamma\left({3 \over 4}\right) {{p_n(s) \Gamma\left({{s+\varepsilon} \over 2} \right)} \over {\Gamma\left({s \over 2}+{{2n+3} \over 4}\right)}},
\ee
where $\varepsilon=0$ for $n$ even and $=1$ for $n$ odd.
The recursions (\ref{eq:1.2a}) and (\ref{eq:1.2b}) then follow by inserting the form (\ref{eq:2.1}) into the recursion for $M_n^\lambda(s)$ given in part (a) of Theorem \ref{theorem2}, and repeatedly applying
the functional equation of the Gamma function, $\Gamma(z+1)=z\Gamma(z)$.

(b) We give a proof which makes use of properties of the Gamma function. From Lemma \ref{lemma7}(a), up to factors not involving $s$, the polynomials $p_n$ may be taken as
\be
\label{eq:2.2}
p_n(s)=(n+1){{\Gamma\left({3 \over 4}\right)} \over 2}{{\Gamma\left({{n+s} \over 2}\right)} \over {\Gamma\left({{s+\varepsilon} \over 2}\right)}} ~_3F_2\left({3 \over 4},{{1-n} \over 2},-{n \over 2};{3 \over 2},1-{{(n+s)} \over 2};1\right).
\ee
From the form of the
numerator parameters, the degree of $p_n$ is evident.

By a `Beta transformation' \cite{grad} (p. 850) we have the following integral
representation:
$$~_3F_2\left({3 \over 4},{{1-n} \over 2},-{n \over 2};{3 \over 2},1-{{(n+s)} \over 2};1\right)$$
$$={\sqrt{\pi} \over {2\Gamma^2(3/4)}}\int_0^1 (1-x)^{-1/4}x^{-1/4} ~_2F_1\left({{1-n} \over 2},-{n \over 2};1-{{(n+s)} \over 2};x\right)dx.$$
We use (\ref{eq:2.2}), together with an $x \to 1-x$ transformation of the $_2F_1$ function
\cite{grad} (p. 1043).  Owing to the poles of the $\Gamma$ function, and that $n$ is a
nonnegative integer, the $_2F_1(x)$ function then transforms to a single $_2F_1(1-x)$ function, and there results
$$p_n(s)={{(n+1)} \over 4} {{\Gamma\left({{n+s} \over 2}\right)} \over {\Gamma\left({{s+\varepsilon} \over 2}\right)}}{\sqrt{\pi} \over {\Gamma(3/4)}}
{{\Gamma\left(1-{{(n+s)} \over 2}\right)} \over {\Gamma\left(1-{s \over 2}\right)}}
{{\Gamma\left({{1+n-s} \over 2}\right)} \over {\Gamma\left({{1-s} \over 2}\right)}}$$
$$ \times \int_0^1 (1-x)^{-1/4}x^{-1/4} ~_2F_1\left({{1-n} \over 2},-{n \over 2};{{s-n+1}
\over 2};1-x\right)dx$$

$$={{(n+1)} \over 4} {\pi \over {\Gamma\left({{s+\varepsilon} \over 2}\right)}}{\sqrt{\pi} \over {\sin\pi\left({{n+s} \over 2}\right)\Gamma(3/4)}}
{1 \over {\Gamma\left(1-{s \over 2}\right)}}
{{\Gamma\left({{1+n-s} \over 2}\right)} \over {\Gamma\left({{1-s} \over 2}\right)}}$$
$$ \times \int_0^1 (1-x)^{-1/4}x^{-1/4} ~_2F_1\left({{1-n} \over 2},-{n \over 2};{{s-n+1}
\over 2};x\right)dx.$$
The following observations lead to verification of the functional equation.
When $n$ is even, $\varepsilon=0$,
$$\Gamma\left({s \over 2}\right)\Gamma\left(1-{s \over 2}\right)={\pi \over {\sin\pi(s/2)}},$$
leaving the denominator factor $\Gamma\left({{1-s} \over 2}\right)$.
When $n$ is odd, $\varepsilon=1$,
$$\Gamma\left({{s+1} \over 2}\right)\Gamma\left({{1-s} \over 2}\right)={\pi \over {\cos\pi(s/2)}},$$
leaving the denominator factor $\Gamma\left(1-{s \over 2}\right)$.

Hence the factor $(-1)^{\lfloor n/2 \rfloor}$ emerges as $\sin (\pi s/2)/\sin [\pi(n+s)/2] =(-1)^{n/2}$ when $n$ is even and as $\cos (\pi s/2)/\sin [\pi(n+s)/2]=(-1)^{(n-1)/2}$ when
$n$ is odd, and the functional equation of $p_n(s)$ follows.

(c) The fact that the zeros all lie on the critical line $\Re s=1/2$ is a corollary of Theorem~\ref{theorem2}~(e).  


\end{proof}

\begin{proof}[Proof of Theorem \ref{theorem3}]
The binomial forms of $M_n^\lambda (s)$ are obtained by rewriting the hypergeometric forms given in Theorem \ref{thm:gr6} in terms of Pochhammer symbols, and then (in a variety of orders) applying the five identities
\[
\binom{a+s-1}{s}=\frac{a(a+1)\ldots (a+s-1)}{s!}=\frac{(a)_s}{s!},
\]
\[
\binom{n+m}{k}\binom{n}{k}^{-1}=\frac{(n+m)_m}{(n-k+1)_m},\qquad
\binom{n}{k+m}\binom{n}{k}^{-1}=\frac{(n-m-k+1)_m}{(k+1)_m},
\]
\[
\binom{n}{k-m}\binom{n}{k}^{-1}=\frac{(n-k+1)_m}{(k-m+1)},\qquad
\binom{c}{b}\binom{a+b}{b}^{-1}=\binom{a+c}{c}^{-1}\binom{a+c}{c-b}.
\]
Corollary \ref{corollary1} then follows immediately by by replacing $M_0^\lambda(s)$ and $M_0^\lambda(s+1)$ with their equivalent binomial coefficients forms when $s$ is respectively an even or an odd integer.
\end{proof}

\begin{remark} With a simple change of variable, the Mellin transforms of $(\ref{eq:m9})$ may
be obtained from \cite{grad} (p. 830).  Alternative forms of these transforms and their
generating functions may be realised by using various expressions from \cite{rainville}
(pp. 279--280).
\end{remark}

\begin{proof}[Proof of Theorem \ref{theorem4}] It follows from either part (c) of Theorem \ref{theorem2} or the hypergeometric form in Theorem \ref{theorem3}, that the Mellin transforms are of the form
$$M_n^\lambda(s)={{\Gamma\left({\lambda \over 2}+{1 \over 4}\right)\Gamma\left({{s+\epsilon}
\over 2}\right)} \over {2(n!)\Gamma\left({{s+n+\lambda} \over 2}+{1 \over 4}\right)}}
p_n^\lambda(s).$$

To demonstrate the difference equation for $p_n^\lambda(s)$
$$[6-4(\lambda+2\lambda n+n^2)-16s+8s(s+1)]\left({{s+\epsilon} \over 2}-1\right)
\left({{s+n+\lambda} \over 2}+{1 \over 4}\right)p_n^\lambda(s)$$
$$+[-9+4(n+\lambda)^2-4(s-1)(s+2)]\left({{s+\epsilon} \over 2}\right)\left({{s+\epsilon} \over 2}-1\right)p_n^\lambda(s+2)$$
$$-4(s-1)(s-2)\left({{s+n+\lambda} \over 2}+{1 \over 4}\right)\left({{s+n+\lambda} \over 2}
-{3 \over 4}\right)p_n^\lambda(s-2)=0,$$
wherein $\epsilon=0$ for $n$ even and $=1$ for $n$ odd, as follows from either part (d) or Theorem \ref{theorem3}, the Mellin transforms are of the form
$$M_n^\lambda(s)={{\Gamma\left({\lambda \over 2}+{1 \over 4}\right)\Gamma\left({{s+\epsilon}
\over 2}\right)} \over {2(n!)\Gamma\left({{s+n+\lambda} \over 2}+{1 \over 4}\right)}}
p_n^\lambda(s).$$
Noting that the factor $\Gamma\left({\lambda \over 2}+{1 \over 4}\right)/(2n!)$ is
independent of $s$, and repeatedly applying the functional equation $\Gamma(z+1)=z\Gamma(z)$, the result follows.

\end{proof}

\begin{proof}[Proof of Theorem \ref{theorem5}] (a) The S:4/1 type combinatorial expressions for the polynomial factors $p_n^\lambda(s)$ are obtained from the S:4/2 type expressions for $M_n^\lambda(s)$ in Theorem \ref{theorem3}, by respectively multiplying through by the factors
\[
\frac{n!(2n)!}{2M_0^\lambda(s)}\binom{n+\frac{s+\lambda }{2}-\frac{3}{4}}{n},\qquad \text{\rm or}\qquad
\frac{n!(2n+1)!}{2M_0^\lambda(s+1)}\binom{n+\frac{s+1+\lambda }{2}-\frac{3}{4}}{n},
\]
depending on whether $n$ is odd or even.

\noindent (b) The two expressions for $q_n^\lambda(s)$ given can be verified by inserting the explicit expressions for $p_n^\lambda$ given in part (a) into the latter expression for $q_n^\lambda(s)$ given in part (b) and rearranging.
The degree of both numerator and denominator polynomials of $q_n^\lambda(s)$ being $\lfloor n/2\rfloor$, then follows from the degree of the polynomials $p_n^\lambda(s)$ given in Theorem \ref{theorem2}, and the number of $s$-linear factors appearing in the denominator product of $q_n^\lambda(s)$.

\noindent (c) The zeros of the denominator polynomials (and so poles of $q_n^\lambda$), correspond to the zeros of the linear factors $2s+2\lambda+4j-3$, or $2s+2\lambda+4j-1$, with $1\leq j\leq \lfloor n/2\rfloor$. For $\lambda >-1/2$, $\lambda\ne 0$ and $\Re\,s>0$, each linear factor is non-zero, ensuring that the rational function $q_n^\lambda(s)$ has no singularities. Hence the `critical zeros' of the polynomials $p_n^\lambda(s)$, are the same as for $q_n^\lambda(s)$, and so for $t\in\mathbb{R}$, the roots of $p_n^\lambda(1/2+it)$ and $q_n^\lambda(1/2+it)$ are identical.

To see that the rational functions $q_n^\lambda(s)$ are normalised with limit 1 as $s\rightarrow \infty$, we consider the limit term by term as $s\rightarrow \infty$ in the S:3/2 sums of $(\ref{eq:crite})$ and $(\ref{eq:crito})$, giving
\[
\mathop{\lim}_{s\rightarrow\infty}
\sum _{r=0}^n \frac{ (-1)^{n-r}
2^{2 r-1}\binom{n+r+\lambda-1}{r}\binom{n+r}{2r}\binom{\frac{1}{2}(s-2)+r}{r}}{\binom{n+r}{r}\binom{\frac{1}{2}(s+\lambda)-\frac{3}{4}+r}{r}}
=\sum _{r=0}^n \frac{ (-1)^{n-r}
2^{2 r-1}\binom{n+r+\lambda-1}{r}\binom{n+r}{2r}}{\binom{n+r}{r}},
\]
\[
\mathop{\lim}_{s\rightarrow\infty}
\sum _{r=0}^n \frac{ (-1)^{n-r}
2^{2 r}\binom{n+r+\lambda}{r}\binom{n+r+1}{2r+1}\binom{\frac{1}{2}(s-1)+r}{r}}{\binom{n+r+1}{r}\binom{\frac{1}{2}(s+\lambda)-\frac{1}{4}+r}{r}}
=\sum _{r=0}^n \frac{ (-1)^{n-r}
2^{2 r}\binom{n+r+\lambda}{r}\binom{n+r+1}{2r+1}}{\binom{n+r+1}{r}}.
\]
Applying the combinatorial identities
\[
\sum _{r=0}^n \frac{ (-1)^{n-r}
2^{2 r-1}\binom{n+r+\lambda-1}{r}\binom{n+r}{2r}}{\binom{n+r}{r}}=\frac{1}{2}\binom{2n+2\lambda-1}{2n-1}\binom{n+\lambda-1}{n-1}^{-1},
\]
\[
\sum _{r=0}^n \frac{ (-1)^{n-r}
2^{2 r}\binom{n+r+\lambda}{r}\binom{n+r+1}{2r+1}}{\binom{n+r+1}{r}}
=\frac{n+1}{2n+1}\binom{2n+2\lambda}{2n}\binom{n+\lambda}{n}^{-1},
\]
we then have the upper bounds for the combinatorial sums, so that $\lim_{s\rightarrow \infty}q_n^\lambda(s)=1$ from below, as required.
The functional equation follows from that for $p_n^\lambda(s)$, by considering the third and fourth displays in part (b) of the theorem.

To see Corollary \ref{corollary40}, substituting $\lambda=1$ in the S/3:2 forms for $p_n^\lambda(s)$, simplifies the sums to the Gould S/2:1 combinatorial functions stated.
Term-by-term analysis of the $n+1$ terms in each sum then reveals that for $s$ an integer, each term is an even integer apart form the $r=0$ term, given by
\[
\frac{n!(2n)!}{2}\binom{n+\frac{s}{2}-\frac{1}{4}}{n},\quad\text{and}\quad
\frac{n!(2n+2)!}{2}\binom{n+\frac{s}{2}+\frac{1}{4}}{n},
\]
depending of whether $n$ is respectively even or odd. The binomial coefficients contribute the power of two $2^{-2n}$, so that the power of 2 in the $r=0$ term is determined by $(2n)!/2^{2n+1}$ when $n$ is even, and $(2n+2)!/2^{2n+1}$, when $n$ is odd. Noting that the $n!$ terms cancel between numerator and denominator, we see that multiplying through by the reciprocal of these respective powers of 2 will produce odd integer values for the $r=0$ term, whilst leaving the others terms $r=1,2,\ldots,n$ even. Hence the summation results in integers having only odd prime factors, being the sum of $n$ even numbers and one odd number.

Analysis of the $n$th Catalan number $$\mathcal{C}_n=\frac{1}{n+1}\binom{2n}{n},$$ shows that the power of 2 in the $\mathcal{C}_n$ is determined by $2^{2n+1}/(2n+2)!$ (A048881 in the OEIS) so that $4\mathcal{C}_{n-1}$ and $\mathcal{C}_n$ have the respective reciprocal powers of 2 to $p_{2n}$ and $p_{2n+1}$. It follows that for $s\in\mathbb{Z}$ we have $4\mathcal{C}_{n-1}p_{2n}$ and $\mathcal{C}_{n}p_{2n+1}$ are odd integers. A slight modification of this argument also removes the odd factors arising in the $(2n)!$ and $(2n+1)!$ polynomial factors such that generated by
\[
\frac{2^{2n+1}}{(2n)!}p_{2n}(s),\quad\text{and}\quad \frac{2^{2n+1}\mathcal{T}_{n+1}}{(2n+2)!}p_{2n+1}(s),
\]
where $\mathcal{T}_{n+1}$ is the largest odd factor of $n+1$. Therefore the above two expressions yield odd integers with fewer prime factors than $4\mathcal{C}_{n-1}p_{2n}$ and $\mathcal{C}_{n}p_{2n+1}$, as required.
\end{proof}

\begin{proof}[Proof of Theorem \ref{theorem6} and Corollary \ref{corollary2}] The proof follows that of Theorem \ref{theorem1}, noting the integral representation
$$\int_0^1 (1-x)^{-\beta}x^{-\beta} ~_2F_1\left({{1-n} \over 2},-{n \over 2};1-{{(n+s)} \over 2};x\right)dx$$
$$ =2^{2\beta-1} {{\sqrt{\pi} \Gamma(1-\beta)} \over {\Gamma(3/2-\beta)}} ~_3F_2\left(1-\beta,{{1-n} \over 2},-{n \over 2};2(1-\beta),1-{{(n+s)} \over 2};1\right),
$$
with $\beta<1$.
The $_2F_1(x)$ function is again transformed to a $_2F_1(1-x)$ function and the other
steps are very similar to before.

The location of the zeros follows from Theorem \ref{theorem2}, setting
$\lambda =3/2-2\beta$.

To see Corollary \ref{corollary2} (a) The initial $\beta=0$ reduction of Theorem \ref{theorem6} to $_2F_1$ form follows from the series definition of the $_3F_2$ function with a shift of summation index and the relations $(1)_j/(2)_j=1/(j+1)$ and $(\kappa)_{j-1}=(\kappa-1)_j/(\kappa-1)$.  The second reduction is a consequence of Gauss summation.
(b) Similarly, with $m$ a positive integer, $(m+1)_j/(2(m+1))_j$ may be reduced and partial
fractions applied to this ratio.  Then with shifts of summation index, the $_3F_2$ function may
be reduced to a series of $_2F_1(1)$ functions.  These in turn may be written in terms of
ratios of Gamma functions from Gauss summation.

\end{proof}


\section{Discussion}

Given the Gould variant combinatorial expressions obtained for $p_n^\lambda (s)$ and $q_n^\lambda (s)$, our results invite several other research questions, such as: is there a combinatorial
interpretation of $p_n^\lambda(s)$ or $q_n^\lambda (s)$, and more generally, of $p_n(s;\beta)$?  Relatedly, is there a
reciprocity relation for $p_n(s)$ and $p_n(s;\beta)$?

Two instances when the combinatorial sums produce ``nice'' combinatorial expressions are
\[
q_{2n}^\lambda(1)=\sum _{r=0}^n \frac{ (-1)^{n-r} 2^{2r-1}\binom{n+r+\lambda-1}{r}\binom{n+r}{2r}\binom{r-\frac{1}{2}}{r}}{\binom{n+r}{r}\binom{\frac{\lambda}{2}-\frac{1}{4}+r}{r}}
={\frac{1}{2}\binom{n+\frac{2 \lambda-3}{4}}{n} \binom{n+\frac{2 \lambda-1}{4}}{n}^{-1}},
\]
\[
q_{2n+1}^\lambda(2)=\sum _{r=0}^n \frac{ (-1)^{n-r}
2^{2 r}\binom{n+r+\lambda}{r}\binom{n+r+1}{2r+1}\binom{r+\frac{1}{2}}{r}}{\binom{n+r+1}{r}\binom{\frac{\lambda}{2}+\frac{3}{4}+r}{r}}
=(n+1) \binom{n+\frac{2 \lambda-3}{4}}{n}\binom{n+\frac{2 \lambda+3}{4}}{n}^{-1}.
\]
In fact, polynomials with only real zeros commonly arise in combinatorics and elsewhere.
Two examples are Bell polynomials \cite{harper} and Eulerian polynomials \cite{comtet} (p. 292). This suggests that there may be other relations of our results to discrete
mathematics and other areas.  In this regard, we mention matching polynomials with only
real zeros.  The matching polynomial $M(G,x)=\sum_k (-1)^k p(G,k)x^k$ counts the matchings
in a graph $G$.  Here, $p(G, k)$ is the number of matchings of size $k$, i.e., the
number of sets of $k$ edges of $G$, no two edges having a common vertex.
$M(G,x)$ satisfies recurrence relations and has only real zeros and $M(G - {v}, x)$
interlaces $M(G,x)$ for any $v$ in the vertex set of $G$ (e.g., \cite{godsil}).  As regards
\cite{bumpchoi,coffeymellin}, we note that
the classical orthogonal polynomials are closely related to the matching polynomials. For
example, the Chebyshev polynomials of the first two kinds are the matching polynomials of paths and cycles respectively, and the Hermite polynomials and the Laguerre polynomials are
the matching polynomials of complete graphs and complete bipartite graphs, respectively.

There are several open topics surrounding the recursions (\ref{eq:1.2a}) and (\ref{eq:1.2b}).  These include:
is it possible to reduce this three-term recursion to two terms, can a pure recursion
be obtained, and, can some form of it be used to demonstrate the occurrence of the
zeros only on the critical line?

The Gegenbauer polynomials have the integral representation
\be
C_n^\lambda(x)={1 \over \sqrt{\pi}}{{(2\lambda)_n} \over {n!}}{{\Gamma\left(\lambda+{1 \over 2}
\right)} \over {\Gamma(\lambda)}} \int_0^\pi (x+\sqrt{x^2-1}\cos \theta)^n \sin^{2\lambda-1}
\theta ~d\theta.
\label{eq:6.1}
\ee
Then binomial expansion of part of the integrand of $M_n^\lambda(s)$ is another way to
obtain this Mellin transform explicitly.  The representation (\ref{eq:6.1}) is also convenient for
showing further special cases that reduce in terms of Chebyshev polynomials $U_n$ or
Legendre or associated Legendre polynomials $P_n^m$.  We mention as examples
\be
C_n^2(x)={1 \over {2(x^2-1)}}[(n+1)xU_{n+1}(x)-(n+2)U_n(x)],
\label{eq:6.2}
\ee
and
$$C_n^{3/2}(x)={{(n+1)} \over {(x^2-1)}}[xP_{n+1}(x)-P_n(x)]=-{{P_{n+1}^1(x)} \over
\sqrt{1-x^2}}.$$

The Gegenbauer polynomials are a special case of the two-parameter Jacobi polynomials
$P_n^{\alpha,\beta}(x)$ (e.g., \cite{andrews}) as follows:
$$C_n^\lambda(x)={{(2\lambda)_n} \over {\left(\lambda+{1 \over 2}\right)_n}} P_n^{\lambda-1/2,
\lambda-1/2}(x).$$
The Jacobi polynomials are orthogonal on $[-1,1]$ with respect to the weight function
$(1-x)^\alpha (1+x)^\beta$.  Therefore, it is also of interest to consider Mellin transforms
such as
$$M_n^{\alpha,\beta}(s)=\int_0^1 x^{s-1}P_n^{\alpha,\beta}(x)(1-x)^{\alpha/2-1/2}(1+x)^{\beta/2
-1/2}dx,$$
especially as the Jacobi polynomials can be written in the binomial form
\[
P_n^{\alpha,\beta}(x)=\sum_{j=0}^n \binom{n+\alpha}{k}\binom{n+\beta}{n-s}\left (\frac{x-1}{2}\right )^{n-s}\left (\frac{x+1}{2}\right )^{s}.
\]
In fact this line of enquiry may provide a far more general approach to investigate `critical polynomials' arising from combinatorial sums.

\section*{Appendix}
Below are collected various transformations of terminating $_3F_2(1)$ series \cite{bailey}, where
again $(a)_n$ denotes the Pochhammer symbol.

$$_3F_2(-n,a,b;c,d;1)={{(c-a)_n(d-a)_n} \over {(c)_n(d)_n}} ~_3F_2(-n,a,a+b-c-d-n+1;
a-c-n+1,a-d-n+1;1)$$
$$={{(a)_n(c+d-a-b)_n} \over {(c)_n(d)_n}} ~_3F_2(-n,c-a,d-a;1-a-n,c+d-a-b;1)$$
$$={{(c+d-a-b)_n} \over {(c)_n}} ~_3F_2(-n,d-a,d-b;d,c+d-a-b;1)$$
$$=(-1)^n{{(a)_n(b)_n} \over {(c)_n(d)_n}} ~_3F_2(-n,1-c-n,1-d-n;1-a-n,1-b-n;1)$$
$$=(-1)^n {{(d-a)_n(d-b)_n} \over {(c)_n(d)_n}} ~_3F_2(-n,1-d-n,a+b-c-d-n+1;a-d-n+1,b-d-n+1;1)$$
$$={{(c-a)_n} \over {(c)_n}} ~_3F_2(-n,a,d-b;d,a-c-n+1;1)$$
$$={{(c-a)_n(b)_n} \over {(c)_n(d)_n}} ~_3F_2(-n,d-b;1-c-n;1-b-n,a-c-n+1;1).$$



\begin{thebibliography}{99}
\bibitem{nbs}M. Abramowitz and I. A. Stegun,
{Handbook of Mathematical Functions, Washington, National Bureau of Standards (1964).}
\bibitem{andrews}G. E. Andrews, R. Askey, and R. Roy,
{Special Functions, Cambridge University Press (1999).}
\bibitem{askey}R. Askey,
{Continuous Hahn polynomials, J. Phys. A {\bf 18}, L1017-L1019 (1985).}
\bibitem{baez} L. B\'aez-Duarte, A sequential Riesz-like criterion for the Riemann hypothesis, Int. J. Math. Sci., {\bf 21}, 3527-3537 (2005).
\bibitem{bailey}W. N. Bailey,
{Generalized hypergeometric series, Cambridge University Press (1935).}
\bibitem{bumpchoi}D. Bump, K.-K. Choi, P. Kurlberg, and J. Vaaler,
{A local Riemann hypothesis, I, Math. Z. {\bf 233}, 1-19 (2000).}
\bibitem{bumpng}D. Bump and E. K.-S. Ng,
{On Riemann's zeta function, Math. Z. {\bf 192}, 195-204 (1986).}
\bibitem{butzer}P. Butzer and S. Jansche,
{A direct approach to the Mellin transform, J. Fourier Analysis Appls. {\bf 3}, 325-376
(1997).}
\bibitem{chihara} T. S. Chihara, An Introduction to Orthogonal Polynomials, Dover Publications, New York (2011).
\bibitem{choi1}S. Choi, J. W. Chung and K. S. Kim, {Relation between primes and nontrivial zeros in the Riemann hypothesis; Legendre polynomials, modified zeta function and Schr\"odinger equation}, J. Math. Phys., {\bf 53}, 122108 (2012); doi: 10.1063/1.4770050 [correction in JMP {\bf 54}, 019901 (2013)].
\bibitem{HDF} M. W Coffey, J. L Hindmarsh, M. C. Lettington and J. Pryce, {On higher dimensional interlacing Fibonacci sequences, continued fractions and Chebyshev polynomials},  J. Theor. Nom. Bordx. {\bf 29}(2), pp. 369-423 (2017).
\bibitem{coffeymellin}M. W. Coffey,
{Special functions and the Mellin transforms of Laguerre and Hermite functions,
Analysis {\bf 27}, 95-108 (2007).}
\bibitem{coffeyxi}M. W. Coffey,
{Theta and Riemann xi function representations from harmonic oscillator eigenfunctions,
Phys. Lett. A {\bf 362}, 352-356 (2007).}
\bibitem{coffeylettunpub}M. W. Coffey and M. C. Lettington,
{Mellin transforms with only critical zeros:  Legendre functions, J. Number
Th., {\bf 148}, 507-536 (2015).}
\bibitem{coffeygen}M. W. Coffey,
{Mellin transforms with only critical zeros: generalized Hermite functions, arXiv:1308.6821
(2013).}
\bibitem{comtet}L. Comtet,
{Advanced combinatorics, Reidel, Dordrecht (1974).}
\bibitem{dilcher} K. Dilcher and K. B. Stolarsky, {A Pascal-Type Triangle Characterizing Twin Primes, Amer. Math. Monthly
{\bf 112}, 673-681 (2005).}
\bibitem{godsil} C. D. Godsil and I. Gutman,
{On the theory of the matching polynomial, J. Graph Theory {\bf 5}, 137-144 (1981).}
\bibitem{grad} I. S. Gradshteyn and I. M. Ryzhik,
{Table of Integrals, Series, and Products, Academic Press, New York (1980).}
\bibitem{gould} H. W. Gould, \emph{Combinatorial Identities}, Morgantown, W. Va. (1972).
\bibitem{harper} L. H. Harper,
{Stirling behavior is asymptotically normal, Ann. Math. Statist. {\bf 38}, 410-414 (1967).}
\bibitem{kirsch} P. Kirschenhofer, A. Peth\"{o}, and R. F. Tichy,
{On analytical and diophantine properties of a family of counting polynomials,
Acta Sci. Math. (Szeged) {\bf 65}, 47-59 (1999).}
\bibitem{koch} H. von Koch, {Sur la distribution des nombres premiers}, Acta Math., {\bf 24}, 159-182, (1901).
\bibitem{maslanka} K. Ma\'slanka, {Hypergeometric-like representation of the zeta-function of Riemann.}, arXiv:math-ph/0105007v1, 2001.
\bibitem{mason} J. C. Mason and D. C. Handscomb,
{Chebyshev polynomials, Chapman \& Hall (2003).}
\bibitem{rainville} E. D. Rainville,
{Special functions, Macmillan (1960).}
\bibitem{rivlin} T. J. Rivlin,
{Chebyshev polynomials:  From approximation theory to algebra and number theory, John Wiley
(1990).}
\bibitem{szego} G. Szeg\"o, Orthogonal Polynomials, AMS Colloquium Publications, Vol. 23, Amer. Math. Soc., Providence RI, (1975).
\end{thebibliography}
\end{document}